\documentclass[10pt]{article}
\title{Splitting Proximal Point Algorithms for the Sum of Prox-Convex Functions} 
\author{Jos\'e de Brito\thanks{School of Science, Great Bay University; Great Bay Institute for Advanced Study,
Dongguan 523000, Guangdong Province,  People's Republic of China; and Instituto Federal do Piau\'{\i}, S\~ao Raimundo Nonato, Piau\'{\i}, Brazil. E-mail: jose.brito@ifpi.edu.br. ORCID-ID: 0000-0003-4362-0536
} 
\and
Felipe Lara\thanks{Instituto de Alta investigaci\'on (IAI), Universidad de
Tarapac\'a, Arica, Chile. E-mail: felipelaraobreque@gmail.com; flarao@academicos.uta.cl. 
Web: felipelara.cl, ORCID-ID: 0000-0002-9965-0921}
\and
Tran Van Thang\thanks{Faculty of Natural Sciences, Electric Power University, Hanoi, Vietnam. E-mail: thangtv@epu.edu.vn, ORCID-ID: 0000-0003-1679-600X}
}
\usepackage{amsmath}
\usepackage{amsthm}
\usepackage{amssymb}
\usepackage{multicol}
\usepackage[active]{srcltx}
\usepackage{algorithm}
\usepackage{graphicx}
\usepackage{xcolor}
\usepackage{amsmath, amssymb}
\usepackage{algorithm}
\usepackage{algpseudocode}
\usepackage{verbatim}
\usepackage{array}

\numberwithin{equation}{section}

\newtheorem{lemma}{Lemma}[section]
\newtheorem{theorem}{Theorem}[section]

\newtheorem{proposition}{Proposition}[section]
\newtheorem{example}{Example}[section]
\newtheorem{definition}{Definition}[section]
\newtheorem{remark}{Remark}[section]

\DeclareMathOperator{\px}{Prox}

\DeclareMathOperator{\amin}{\arg\min}

\begin{document}

\maketitle

\begin{abstract}

\noindent {\bf Abstract}
This paper addresses the minimization of a finite sum of prox-convex functions under Lipschitz continuity of each component. We propose two variants of the splitting proximal point algorithms proposed in \cite{Bacak,Bertsekas}: one deterministic with a fixed update order, and one stochastic with random sampling, and we extend them from convex to prox-convex functions. We prove global convergence for both methods under standard stepsize a\-ssump\-tions, with almost sure convergence for the stochastic variant via supermartingale theory. Numerical experiments with nonconvex quadratic functions illustrate the
efficiency of the proposed methods and support the theoretical results.

\medskip

{\small}


\medskip

\end{abstract}

\section{Introduction}

Nonconvex optimization is widely present in modern machine learning and scientific computing. However, it remains substantially more challenging than its convex counterpart, primarily due to the prevalence of suboptimal local minima and the general lack of global optimality guarantees.

A common strategy to bridge this gap is to identify structured classes of nonconvex functions that retain tractable properties. This pursuit has given rise to the field of generalized convexity. Notions such as quasiconvexity, pseudoconvexity, weak convexity, and the difference-of-convex (DC) structure have long been studied to relax convexity while preserving manageable analysis and convergence guarantees \cite{CMA,HKS}.

In this work, we focus on the class of prox-convex functions, a recent ge\-ne\-ra\-li\-za\-tion introduced in \cite{GL1}. We study the problem of minimizing a finite sum of such functions over a closed set $K \subseteq \mathbb{R}^{n}$:
\begin{equation}\label{main}
\min_{x \in K} f(x) := \sum_{i=1}^{N} f_i(x),
\end{equation}
where $f_i: K \to \mathbb{R}$ is a Lipschitz continuous and prox-convex function for every $i \in \{1, \ldots, N\}$, and $N \in \mathbb{N}$.

The appeal of prox-convexity stems from its preservation of key operational properties found in convex optimization, specifically, each function $f_i$ has a single-valued and firmly nonexpansive proximal operator \cite{GL1}, a trait tra\-di\-tio\-na\-lly exclusive to convex functions. However, a significant challenge arises: the aggregate function $f$ is not necessarily prox-convex, and its proximal operator generally lacks a closed-form expression. Consequently, classical proximal point algorithms, which require solving the full proximal subproblem at each iteration, are often computationally intractable for $f$. This limitation necessitates the use of splitting techniques that instead leverage the tractable proximal operators of the individual components $f_i$.

While the proximal point algorithm (PPA henceforth) \cite{Rock76} minimizes a convex function by iteratively applying its proximal operator, splitting methods are designed for composite problems. These algorithms decompose the objective by utilizing the proximal operators of its individual parts, thus avoiding the prohibitive cost of computing the proximal operator of the entire sum. This methodology, pioneered for operators in \cite{PL}, has become a cornerstone of modern first-order optimization, as extensively detailed in contemporary texts \cite{BC-2, B}.

When the objective function $f$ decomposes as a sum of many terms, a natural strategy is to employ incremental proximal steps on each component. This idea was formalized for the convex setting in \cite{Bertsekas}, who proposed an incremental proximal method for large-scale optimization. Their algorithm, often termed a cyclic proximal point algorithm, sequentially applies the proximal operator of each $f_i$ in a fixed order, thus avoiding the need for a single, costly proximal step on the entire sum $f$. This cyclic scheme is guaranteed to converge for convex problems under standard assumptions~\cite{B,Bertsekas}. A key advantage of such methods is their decomposition of a large, complex problem into $N$ simpler subproblems, each involving only a single component.

In our prox-convex setting, adopting a similar splitting strategy is highly natural, as each $f_i$ can be efficiently minimized in isolation via its proximal operator. The central question, however, is whether the favorable convergence properties of these methods persist without convexity. Recent research provides encouraging evidence. An extension of PPA for prox-convex functions was established in~\cite{GL1}, and subsequent works have successfully incorporated advanced features like relaxation and inertia into proximal schemes for structured nonconvex problems \cite{GLM-3,GMT2025}. These developments demonstrate that the principles of PPA can be generalized to certain nonconvex settings by leveraging structured properties like prox-convexity as a substitute for global convexity.

In this context, we propose a deterministic splitting PPA for problem \eqref{main} (Method 1). 
This algorithm applies the proximity operators of $f_1, \ldots, f_N$ a\-ccor\-ding to a fixed permutation at each cycle, encompassing the purely cyclic case as a particular instance, and can be seen as an extension of the incremental strategy of \cite{Bertsekas} to prox-convex functions. Furthermore, we study a randomized splitting PPA (Method~2), where at each iteration a single index $i\in\{1,\ldots,N\}$ is chosen at random and a proximal step is taken only with respect to $f_i$. 
Such random-order schemes are well established in the convex literature and are attractive due to their simplicity and low per-iteration cost \cite{BT89,NB01,NB03,RT14}. 
We prove that, under standard diminishing step-size conditions, both algorithms converge to a minimizer of $f$, despite the lack of global convexity. 
In the deterministic case, convergence follows from a descent pro\-per\-ty, while in the stochastic case almost sure convergence is established using supermartingale arguments~\cite{Neveu-75}. Numerical experiments on nonconvex quadratic problems illustrate the practical relevance of our theoretical results.

The structure of the paper is as follows. Section \ref{sec:02} reviews preliminaries and definitions of prox-convex functions and their properties. Section \ref{sec:03} presents the deterministic (Method~1) and stochastic (Method~2) splitting algorithms and its convergence analysis. Section \ref{sec:04} contains numerical experiments, and Section \ref{sec:05} concludes with remarks and perspectives.

\section{Preliminaries and Basic Definitions}\label{sec:02}

The inner product of $\mathbb{R}^{n}$ and the Euclidean norm are denoted by $\langle \cdot,\cdot \rangle$ and $\lVert \cdot \rVert$, respectively. Let 
$K$ be a nonempty set in $\mathbb{R}^{n}$, its closure is denoted by $\overline{K}$, its boundary by ${\rm bd}\,K$, its topological interior by ${\rm 
int}\,K$ and its convex hull by ${\rm conv}\,K$.  Given a convex and closed 
set $K$, the projection of $x$ on $K$ is denoted by $P_{K} (x)$, and the 
indicator function on $K$ by $\iota_{K}$. The ball with cen\-ter at $x$ and 
radius $\delta > 0$ is denoted by $\mathbb{B} (x, \delta)$ and the identity operator by $I$.

Given any $x, y, z \in \mathbb{R}^{n}$ and any $\beta \in \mathbb{R}$, the 
following relations hold: 
\begin{align}
 & ~~~~~ \langle x - z, y - x \rangle= \frac{1}{2} \lVert z - y \rVert^{2} -
 \frac{1}{2} \lVert x - z \rVert^{2} - \frac{1}{2} \lVert y - x \rVert^{2}, 
 \label{3:points} \\
 & \lVert \beta x + (1-\beta) y \rVert^{2} = \beta \lVert x \rVert^{2} + (1 -
 \beta) \lVert y\rVert^{2} - \beta(1 - \beta) \lVert x - y \rVert^{2}. 
 \label{iden:1}
\end{align}

Given any extended-valued function $h: K\ \rightarrow
\overline{\mathbb{R}} := \mathbb{R} \cup \{\pm \infty\}$, the effective 
domain of $h$ is defined by ${\rm dom}\,h := \{x \in K: 
h(x) < + \infty \}$. It is said that $h$ is proper if ${\rm dom}\,h$ is 
nonempty and $h(x) > - \infty$ for all $x \in K$. The notion of pro\-per\-ness is important when dealing with minimization pro\-blems.

It is indicated by ${\rm epi}\,h := \{(x,t) \in K \times
\mathbb{R}: h(x) \leq t\}$ the epigraph of $h$, by $S_{\lambda} (h) := 
\{x \in K: h(x) \leq \lambda\}$ the sublevel set of $h$ at 
the height $\lambda \in \mathbb{R}$ and by ${\rm argmin}_{K}  h$ the set of all minimal points of $h$ in $K$. Furthermore, the 
current convention $\sup_{\emptyset} h := - \infty$ and $\inf_{\emptyset} 
h := + \infty$ is adopted.

A function $h$ defined on $K$ with a convex domain  is said to be
\begin{itemize}
 \item[$(a)$] (strongly) convex on ${\rm dom}\,h$, if there exists $\gamma \geq 0$ such that for all $x, y \in \mathrm{dom}\,h$ and all $\lambda \in [0, 1]$, we have
 \begin{equation}\label{strong:convex}
  h(\lambda y + (1-\lambda)x) \leq \lambda h(y) + (1-\lambda) h(x) -
  \lambda (1 - \lambda) \frac{\gamma}{2} \lVert x - y \rVert^{2},
 \end{equation}

 \item[$(b)$] (strongly) quasiconvex on ${\rm dom}\,h$, if there exists $\gamma \geq 0$ such that for all $x, y \in \mathrm{dom}\,h$ and all $\lambda \in [0, 1]$, we have
 \begin{equation}\label{strong:quasiconvex}
  h(\lambda y + (1-\lambda)x) \leq \max \{h(y), h(x)\} - \lambda(1 -
  \lambda) \frac{\gamma}{2} \lVert x - y \rVert^{2}.
 \end{equation}
\end{itemize} 
\noindent Every (strongly) convex function is (strongly) quasiconvex, while the reverse statement does not holds (see \cite{CMA,HKS,Lara-9}). 

A proper function $h: K \rightarrow \overline{\mathbb{R}}$ is said to be (see \cite{CC}):
\begin{itemize}
 \item[$(i)$] 2-supercoercive, if
 \begin{equation}
  \liminf_{\lVert x \rVert \rightarrow+ \infty} \frac{h(x)}{\lVert x
  \rVert^{2}} >0,
 \end{equation}

%
%
 \item[$(ii)$] coercive, if
 \begin{equation}
  \lim_{\lVert x \rVert \rightarrow+ \infty} h(x) = + \infty.
 \end{equation}
 or equivalently, if $S_{\lambda} (h)$ is bounded for all $\lambda \in
 \mathbb{R}$.

 \item[$(iii)$] $2$-weakly coercive, if 
  \begin{equation}\label{2weakly:coercive}
  \liminf_{\, \lVert x \rVert \rightarrow + \infty} \frac{h(x)}{
   \lVert x \rVert^{2}} \geq 0,
  \end{equation}
\end{itemize}
Clearly, $(i) \Rightarrow (ii) \Rightarrow (iii)$, but the converse statements 
do not hold in general. Indeed, the function $h(x) = \lvert x \rvert$ is 
coercive without being $2$-supercoercive while the function $h(x) = x$ is 
$2$-weakly coercive without being coercive. A survey on coercivity notions 
is \cite{CC}.

As noted in \cite[Theorem 1]{Lara-9}, strongly quasiconvex functions are 
$2$-super\-coer\-ci\-ve. In particular, lsc strongly quasiconvex functions 
has an unique global minimum on closed convex sets (see \cite[Corollary 3]{Lara-9}). 


A function $h: K \rightarrow \mathbb{R}$ its said to be 
$L$-smooth if it is differentiable  and
\begin{equation}\label{L:smooth}
 \lVert \nabla h(x) - \nabla h(y) \rVert \leq L \lVert x - y \rVert, ~ 
 \forall ~ x, y \in K.
\end{equation}

Given a proper function $h:K \rightarrow \overline{
\mathbb{R}}$, the convex subdifferential of $h$ at $\overline{x} \in 
{\rm dom}\,h$ is defined by
\begin{equation} \label{subd:usual}
 \partial_K h(\overline{x}) := \{ \xi \in \mathbb{R}^{n}: ~ h(y) \geq
 h(\overline{x}) + \langle \xi, y - \overline{x} \rangle, ~ \forall ~ y \in
 K\},
\end{equation}
and by $\partial_K h(x) = \emptyset$ if $x \not\in {\rm dom}\,h$.


The \textit{proximity operator of parameter $\beta >0$} of a function
$h:\mathbb{R}^n \rightarrow \overline {\mathbb{R}}$ at $x \in
\mathbb{R}^{n}$ is defined as
$ {\rm Prox}_{\beta h}: \mathbb{R}^{n} \rightrightarrows 
 \mathbb{R}^{n}$, where
 \begin{equation}\label{gammah-def}x \mapsto\, \amin\limits_{y \in \mathbb{R}^{n}} 
 \left\{h(y) + \frac{1}{2 \beta} \|y-x\|^2\right\}.
\end{equation}
For any lower semicontinuous (lsc henceforth) function, $2$-weakly coerciveness is sufficient 
for the nonemptiness of the proximity operator (see \cite[Proposition 
3.1]{GL1} for instance). Furthermore, when $h$ is proper, convex and lsc, ${\rm Prox}_{\beta h}$ turns out to be a single-valued 
operator, and the  following well-know identity holds
\begin{align}\label{convex:proxsubd}
 \overline{x} = {\rm Prox}_{\beta h} (z) ~  \Longleftrightarrow ~ 
 \frac{1}{\beta} (z - \overline{x}) \in \partial h(\overline{x}).
\end{align}
 Moreover, when $\beta=1 $ we write ${\rm Prox}_{h}$ instead of 
 ${\rm Prox}_{1 h}$.

We denote
\begin{equation}\label{prx}
 {\rm Prox}_{\beta h} (K, z) := {\rm Prox}_{(\beta h + \iota_{K})} (z) = 
 {\rm argmin}_{y \in K } \left\{h(y) + \frac{1}{2 \beta} \|y-z\|^2\right\}.
\end{equation}

We recall the following generalized convexity notion \cite{GL2,GL1}.

\begin{definition}\label{def:pconvex}
 Let $K$ be a closed set in $\mathbb{R}^{n}$ and $h: \mathbb{R}^{n} \rightarrow \overline{\mathbb{R}}$ be a proper function such that $K \cap {\rm dom}\,h \neq \emptyset$ and $\beta > 0$. It is said that $h$ is \textit{prox-convex on
 $K$} if there exists $\alpha > 0$ such that for every $z \in K$, $\px_{\beta h}
 (K, z) \neq \emptyset$, and
 \begin{equation}\label{prox:all}
  \overline{x} \in {\rm Prox}_{\beta h} (K, z) ~ \Longrightarrow ~ h(\overline{x}) -
  h(x) \leq \frac{\alpha}{\beta} \langle \overline{x} - z, x - \overline{x} \rangle, ~
  \forall ~ x \in K.
 \end{equation}
\end{definition}

 The scalar $\alpha > 0$ for which \eqref{prox:all} holds is said to be
 the \textit{prox-convex value} of the function $h$ on $K$. When $K =
 \mathbb{R}^{n}$ we say that $h$ is \textit{prox-convex}.
 
 \begin{example} (see \cite{T2024})\label{exam2.2}
 We consider the nonconvex functions $g(x) = -x^2-x$ and $h=5x+\ln(10x+1)$, which  were employed  in \cite{GL1}  as cost functions for oligopolistic equilibrium problems. Then,
 \begin{itemize}
  \item[$(i)$] $g(x)$  is prox-convex on $[0,r]$ with modulus $\alpha=1$ and satisfying (\ref{prox:all}) for all $\beta>0$;
 	
 \item[$(ii)$] $h(x)$  is prox-convex on $[1,2]$ with modulus $\alpha=1$ and satisfying (\ref{prox:all}) for all $\beta>\frac{1}{5}$.
 \end{itemize}
\end{example}

Every (strongly) convex function is prox-convex with prox-convex value 
$\alpha = 1$ (see \cite[Proposition 3.4]{GL1}). The reverse statement 
does not hold as the function $h: [0,1] \rightarrow \mathbb{R}$ given by 
$h(x) = -x^{2}-x$ shows (see \cite[Example 3.1]{GL1}).

\begin{lemma} (see \cite{T2024})\label{Lem2.1}
 Let $\mathbb R^n = \prod^{N}_{i=1} \mathbb{R}^{n_{i}}$, where $\sum^{N}_{i=1} n_{i} = n$ and $n_i \geq 1$, with $i=1,\ldots,N$. Then for any $x \in \mathbb R^n$ we have $x=(x^1,\ldots,x^{N})$ where $x^i \in \mathbb R^{n_i}$ for every $i=1,\ldots,N$. For each  $i\in \{1,\ldots,N\}$, let $C_i$ be a closed subset in $\mathbb R^{n_i}$ and $g_i: C_i \to \mathbb R\cup \{+\infty\}$ be proper prox-convex on $C_i$. Then the function $g(x)=\sum_{i=1}^{N} g_i(x^i)$ is prox-convex on $C$, where $C=\prod^{N}_{i=1}C_i$.
\end{lemma}

\begin{lemma}(see \cite{GL1})\label{sqcx:pconvex}
 Let $K$ be a closed and convex subset in $\mathbb R^n$ and $h: 
 \mathbb R^n \to \overline{\mathbb{R}}$ be a proper lsc function on 
 $K$ such that $K \cap {\rm dom}\,h \ne \emptyset$. If $h$ satisfies the 
 following con\-ditions:
 \begin{itemize}
  \item[$(a)$] $h$ is strongly quasiconvex on $K$ with modulus $\eta > 0$,
			
  \item[$(b)$] for each $z \in \mathbb R^n$ there exists $\overline{x}$ such 
  that ${\rm Prox}_{h} (K, z) = \overline{x}$ and 
  \begin{equation*}
   \frac{1}{2} (z - \bar x) \in \partial^{\le} (h +\iota_K) (\overline{x}),
   \end{equation*}   	 
  \end{itemize}
  then $h$ is prox-convex on $K$ with prox-convex value $\alpha = \frac{1}{2}$ (where $\partial^{\le}$ is the Guti\'errez subdifferential \cite{Guti}).           
\end{lemma}

\begin{remark}\label{not:equals}
 Note that strongly quasiconvex functions and prox-convex functions are 
 not related each other. Indeed, the function $x\in \mathbb{R}_+ \mapsto 
 \sqrt{x}$ is strongly quasiconvex on convex and bounded intervals (see 
 \cite[Proposition 16]{Lara-9}) without being prox-convex while the constant function $x \in \mathbb{R}^n \mapsto \alpha 
 \in \mathbb{R}$ is convex (hence prox-convex) without being strongly
 quasiconvex.
\end{remark}

For a prox-convex function $h$ with prox-convex value $\alpha 
> 0$, we have
\begin{equation} \label{equa7}
     \overline{x} \in {\rm Prox}_{h} (K, z) ~ \Longrightarrow ~ z - 
\overline{x} \in \partial \left(\frac{1}{\alpha} (h + \iota_K) \right) (\overline{x}).
\end{equation}

When $K$ is a closed set in $\mathbb{R}^{n}$ and $h: \mathbb{R}^{n}
 \rightarrow \overline{\mathbb{R}}$ is a proper prox-convex function on $K$ such that $K \cap {\rm dom}\,h \neq \emptyset$. Then the map 
 $z \rightarrow \px_{h} (K, z)$ is single-valued and firmly nonexpansive by \cite[Proposition 3.3]{GL1}.

 If $h$ is prox-convex with prox-convex value $\alpha > 0$, then (see \cite[Page 322]{GL1}) we know that $\px_{(1/\alpha) h} = 
 {\rm Prox}_{h}$ is a singleton, hence
\begin{equation}
 ^{\frac{1}{\alpha}}h(z) = \min_{x \in K} \left(h(x) + \frac{\alpha}{2}
 \lVert z - x \rVert^{2} \right) = h ({\px}_{h} (z)) + \frac{\alpha}{2}
 \lVert z - {\px}_{h}(z) \rVert^2.
\end{equation}
Consequently, $^{1/\alpha}h(z)\in \mathbb{R}$ for all $z \in 
\mathbb{R}^{n}$.

The following result will be useful in the sequel.

\begin{lemma}\label{lemma1}
Let  $h: K \to \overline{\mathbb{R}}$ be a function prox-convex in $K$ with modulus $\alpha>0$. 
If $\bar{x} \in {\rm Prox}_{\beta h}(K,z)$, then
\begin{equation}\label{equa6}
h(\bar{x})-h(x)\leq \dfrac{\alpha}{2\beta } \left(\Vert z-x\Vert ^2 - \Vert \bar{x}-x\Vert ^2 \right), 
~ \forall ~ x \in {\rm dom}\,h.
\end{equation}
\end{lemma}

\begin{proof}
 Let $x \in {\rm dom}\,h$. Note that
 \begin{align*}
  & \Vert z - x \Vert^2 = \Vert z-\bar{x} + \bar{x} - x \Vert^2 = \Vert z - \bar{x} \Vert^2 + 2 \langle z - \bar{x}, \bar{x} - x \rangle + \Vert \bar{x} - x \Vert^2 \\
  & \hspace{4.4cm} \geq \Vert z - \bar{x} \Vert^2 + \dfrac{2 \beta}{\alpha} (h(\bar{x}) - h(x))+ \Vert \bar{x} - x \Vert^2 \\
  & \Longrightarrow ~ h(\bar{x}) - h(x) \leq \dfrac{\alpha}{2\beta } \left(\Vert z - x \Vert ^2 - \Vert \bar{x} - x \Vert^2 - \Vert z - \bar{x} \Vert^2 \right),
\end{align*} 
and the result follows.
\end{proof}

For a further study on generalized convexity, prox-convexity and proximal point type methods we refer to \cite{BC-2,B,Boyd11,CMA,GL2,GL1,GLM-3,GMT2025,HKS,ISO2023,ISD2023,J-2,KL-2,Lara-9,LT2024} and references therein.

\section{Splitting Proximal Point
Algorithms}\label{sec:03}

Let $K \subseteq \mathbb{R}^n$ be a nonempty closed set, $N \in \mathbb{N}$ and $f_i : K \to \overline{\mathbb{R}}$ be proper, lsc and prox-convex functions on $K$ for every $i \in \{1,\ldots, N\}$. We are interested in problem \eqref{main}, that is,
\begin{equation}\label{equa1} 
 \min_{x \in K} f(x) := \sum_{i=1}^{N} f_i(x).
\end{equation} 
In the following subsections, we will present two splitting type methods for solving problem \eqref{equa1}.

\subsection{Permuted-order Version}

\paragraph{Method 1 (Permuted-order version).}
Consider a function $f:K \to \overline{\mathbb{R}}$ given by \eqref{equa1}, an initial point $x_0 \in \mathbb{R}^n$ and a stepsize sequence $(\beta_k)_{k \in \mathbb{N}_0}$ of po\-si\-ti\-ve reals. 

For each $k \in \mathbb{N}_0$, let $\pi_k : \{1,\ldots, N\} \to \{1, \ldots, N\}$ be a permutation. Then the sequence $(x_j)_{j \in \mathbb{N}_0}$ is recursively defined as follows:
\begin{equation}\label{permuted method} \begin{aligned}  
x_{kN+j} &\in \text{Prox}_{\beta_k f_{\pi_k(j)}}(K,x_{kN+(j-1)}), \qquad j = 1, \ldots, N. 
\end{aligned} 
\end{equation}

\medskip

We immediately observe the following.

\begin{proposition}\label{prop:permuted}
Let $(x_l)$ be the sequence defined by the permuted-order method \eqref{permuted method}, where the functions $f_1, \ldots, f_N : K \to \overline{\mathbb{R}}$ 
are Lipschitz continuous with a common constant $L$ and prox-convex in $K$ with modulus $\alpha > 0$. 
Then, for any $x \in {\rm dom}\,f $, the following 
inequality holds
\begin{equation}\label{eq:permuted-rate}
 \Vert x_{(k+1)N} - x \Vert^2 \leq \Vert x_{kN}-x\Vert ^2- \dfrac{2\beta_k
 }{\alpha}  (f(x_{kN})-f(x)) + \dfrac{2 L^2N (N+1)\beta_k^2}{\alpha}.
\end{equation}
\end{proposition}

\begin{proof}
Given $k\in \mathbb{N}_0$ and $j \in \{1, \ldots, N\}$, we apply Lemma~\ref{lemma1} to the function $f_{\pi_k(j)}$, using $z = x_{kN+(j-1)}$, $\beta = \beta_k$, and $x \in \operatorname{dom} f$, which yields
\[
\Vert x_{kN+j} - x \Vert^2 \leq \Vert x_{kN+(j-1)} - x \Vert^2 - \dfrac{2\beta_k}{\alpha}[f_{\pi_k(j)}(x_{kN+j}) - f_{\pi_k(j)}(x)].
\]
Fixing $k$ and summing over $j \in \{1, \ldots, N\}$, we obtain
\begin{equation} \label{equa14}
\Vert x_{kN+N} - x \Vert^2 \leq \Vert x_{kN} - x \Vert^2 - \dfrac{2\beta_k}{\alpha} \sum_{j=1}^N [f_{\pi_k(j)}(x_{kN+j}) - f_{\pi_k(j)}(x)].
\end{equation}
By the definition of Method~1, we have
\[
f_{\pi_k(j)}(x_{kN+j}) + \frac{1}{2\beta_k} \Vert x_{kN+j-1} - x_{kN+j} \Vert^2 \leq f_{\pi_k(j)}(x_{kN+j-1}),
\]
for every $j = 1, \dots, N$. Since $f_{\pi_k(j)}$ is Lipschitz continuous with constant $L$, we get
\begin{equation}\label{equa9}
\Vert x_{kN+j} - x_{kN+j-1} \Vert \leq 2\beta_k \frac{f_{\pi_k(j)}(x_{kN+j-1}) - f_{\pi_k(j)}(x_{kN+j})}{\Vert x_{kN+j-1} - x_{kN+j} \Vert} \leq 2\beta _k L.
\end{equation}
Using the Lipschitz continuity of each component function and applying \eqref{equa9}, we obtain
\begin{align*}
 & ~~~\, \sum_{j=1}^N [f_{\pi_k(j)}(x_{kN+j}) - f_{\pi_k(j)}(x)] \\ 
 & = \sum_{j=1}^N \left[ (f_{\pi_k(j)}(x_{kN+j}) - f_{\pi_k(j)}(x_{kN})) + (f_{\pi_k(j)}(x_{kN}) - f_{\pi_k(j)}(x)) \right] \\
 & = f(x_{kN}) - f(x) + \sum_{j=1}^N [f_{\pi_k(j)}(x_{kN+j}) - f_{\pi_k(j)}(x_{kN})] \\
 & \geq f(x_{kN}) - f(x) - L \sum_{j=1}^{N} \Vert x_{kN+j} - x_{kN} \Vert \\
 & \geq f(x_{kN}) - f(x) - L \sum_{j=1}^{N} \left( \Vert x_{kN+j} - x_{kN+j-1} \Vert + \ldots + \Vert x_{kN+1} - x_{kN} \Vert \right) \\
 & \geq f(x_{kN}) - f(x) - L \sum_{j=1}^{N} 2\beta_k L j \\
 & = f(x_{kN}) - f(x) - L^2 \beta_k N (N+1).
\end{align*}
Then, it follows from \eqref{equa14} that
\begin{align*}
 \Vert x_{(k+1)N} - x \Vert^2 \leq \Vert x_{kN}-x\Vert ^2- \dfrac{2\beta_k
 }{\alpha}  (f(x_{kN})-f(x)) + \dfrac{2 L^2N (N+1)\beta_k^2}{\alpha},
\end{align*}
and the result follows.
\end{proof}

\begin{remark}
 Let $x \in \operatorname*{arg\,min}_{y\in K} f(y)$ and suppose the conditions of Proposition \ref{prop:permuted} hold.
If $f(x_{kN})>f(x)$ and
$$ 0 < \beta_k < \frac{f(x_{kN}) - f(x)}{L^2 N(N+1)},$$
then
$$ \Vert x_{(k+1)N} - x \Vert^2 < \Vert x_{kN} - x \Vert^2. $$ Indeed, from \eqref{eq:permuted-rate}, we have
\begin{align*}
 \Vert x_{(k+1)N} - x \Vert^2 & \leq \Vert x_{kN} - x \Vert^2 + \dfrac{2 \beta_k}{\alpha} (\beta_k L^2 N (N+1) - (f(x_{kN}) - f(x))) \\
 & < \Vert x_{kN} - x \Vert^2,
\end{align*} 
which guarantees a strict decrease in the distance to the minimizer after each full cycle of the Method.

\end{remark}

Based on Proposition \ref{prop:permuted}, we   establish that the sequence generated by Method 1 converges to a minimizer of $f$.

\begin{theorem}
Let $ (x_l) $ be the sequence defined in method 1, where $(\beta_k)_{k \in \mathbb{N}_0}$ is a sequence of positive real numbers satisfying

\begin{equation} \sum_{k=0}^{\infty} \beta_k = +\infty, \quad \sum_{k=0}^{\infty} \beta_k^2 < +\infty. \end{equation}
If ${\rm argmin}_{K}  f \neq \emptyset$ and the functions $ f_1, \ldots, f_N $ are Lipschitz continuous with a common constant $L$ and prox-convex with modulus $\alpha>0$, then the sequence $ (x_l) $ converges to a minimizer of $ f $. 
\end{theorem}

\begin{proof}
 Applying Proposition \ref{prop:permuted} for $x \in {\rm argmin}_{K}  f$, it follows that
\begin{equation}\label{equa5}
 \Vert x_{kN+N} - x \Vert^2 \leq\Vert x_{kN} - x \Vert^2 + \dfrac{2L^2N(N+1)\beta_{k}^{2}}{\alpha}. 
\end{equation}
By summing \eqref{equa5} over $k \geq p$, where $p \in \mathbb{N}$, and taking $\limsup_{k \to \infty}$, we have
\begin{equation}
    \displaystyle \limsup_{k \to \infty} \Vert x_{kN} - x \Vert^2  \leq \Vert x_{pN} - x \Vert^2 + \dfrac{2L^2N(N+1)}{\alpha} \sum_{k=p}^{\infty}\beta_k^2.
\end{equation}
Since $\sum_{k=1}^\infty \beta_k^2< \infty$, taking $\liminf_{p \to \infty}$ gives 
$$\displaystyle \limsup_{k \to \infty} \Vert x_{kN} - x \Vert^2  \leq  \displaystyle \liminf_{p \to \infty} \Vert x_{pN} - x \Vert^2. 
$$
Hence, the sequence $\left(\Vert x_{kN} - x \Vert \right)_{k \in \mathbb{N}_0}$ converges.  It follows from \eqref{eq:permuted-rate} that
\begin{equation*}
 \dfrac{2}{\alpha} \sum_{k=1}^{p}[f(x_{kN}) - f(x)] \beta_k \leq   \Vert x_{pN} - x \Vert^2 - \Vert x_{pN+N} - x \Vert^2+ \dfrac{2 L^2N(N+1)}{\alpha}\sum_{k=1}^{p}\beta_k^2. 
\end{equation*}
Thus, $ \sum_{k=1}^{p}[f(x_{kN}) - f(x)]\beta_k<+\infty$. As $\sum_{k=1}^{\infty} \beta_k = +\infty$, there exists a sub\-se\-quen\-ce $(x_{k_jN})_{j \in \mathbb{N}}$ such that 
$$\displaystyle \lim_{j \to \infty} f(x_{k_jN}) = f(x) = \operatorname{min}(f).$$  
Since the sequence $\left(\Vert x_{k_jN} - x \Vert \right)_{k \in \mathbb{N}}$ converges, it follows that $\left(x_{k_jN}\right)$ admits a subsequence converging to some $x^{\ast} \in K$. By the continuity of $f$, we have $f(x^{\ast}) = \operatorname{min}(f)$, implying $x^{\ast} \in {\rm argmin}_{K}  f$. Hence, $\left(\Vert x_{kN} - x^{\ast} \Vert \right)_{k \in \mathbb{N}}$ converges to zero, which means  
$$\displaystyle \lim_{k \to \infty} x_{kN} = x^{\ast}.$$
Given $n \in \{1, \ldots, N\}$, it follows from \eqref{equa9} that 
\begin{align*}
 \Vert x_{kN+n} - x^{\ast}\Vert & \leq \Vert x_{kN+n} - x_{kN+(n-1)} \Vert +\ldots + \Vert x_{kN+1} - x_{kN} \Vert + \Vert x_{kN} - x^{\ast} \Vert \\
 & \leq 2 \beta_k L n + \Vert x_{kN} - x^{\ast} \Vert.
\end{align*}
 Since $\beta_k \to 0$, it follows that $\left( x_{kN+n} \right)$ converges to $x^{\ast}$, and consequently, $(x_l)$ also converges to $x^{\ast}$,  which concludes the proof.
\end{proof}

\begin{remark}
 The results presented in this section also apply when using a cyclic update order. This cyclic method, where proximity operators are applied one after another in a loop, has been studied for convex functions by \cite{Bertsekas,Bacak}. In this case, the proximity operators associated with the functions $f_1, \ldots, f_N$ are applied sequentially in a fixed cyclic order, that is, given a sequence of positive real numbers $(\beta_k)_{k \in \mathbb{N}_0}$ and an initial point $x_0 \in \mathbb{R}^n$, the method generates a sequence $(x_k)$ according to the following rule: for each $k \in \mathbb{N}$,
\begin{align*}
 x_{kN+1} & \in {\rm Prox}_{\beta_k f_1}(x_{kN}), \\
 x_{kN+2} & \in {\rm Prox}_{\beta_k f_2}(x_{kN+1}), \\
 & \vdots \\ 
 x_{(k+1)N} & \in {\rm Prox}_{\beta_k f_N}(x_{(N+N-1}).
\end{align*}
\end{remark}

\begin{remark}
 Our deterministic scheme generalizes the incremental proximal methods of \cite{Bacak, Bertsekas} from the convex to the prox-convex setting. Furthermore, a key difference is that the full function $f$ need not be prox-convex itself, just all components $f_{i}$.
\end{remark}

\subsection{Random Order Version.}

In this section, we study a stochastic variant of the splitting PPA to solve problem~\eqref{equa1}, in which each component function $f_i$ is assumed to be prox-convex in a nonempty closed set in $K$. At each iteration, the index of the function to be processed is randomly selected.

Consider the set $B =\{1, \ldots, N\}$ and the uniform probability measure $p$ on $B$, where $p_i= p(\{i\}) = \frac{1}{N}$ for each $i \in B$. We denote by $\Sigma$ the product space $B^{\mathbb{N}}$, furnished with the product $\sigma$-algebra $\mathcal{A}$ and the product measure $\mathbb{P} = p^{\mathbb{N}}$. This means that each element of $\Sigma$ is an infinite sequence of natural numbers in $\{1, \ldots, N\}$. The $\sigma$-algebra $\mathcal{A}$ of $\Sigma$ is generated by the cylinders 
$$[m; \alpha_m, \ldots , \alpha_k]:= \{(x_i)_{i \in \mathbb{N}}; x_i = \alpha_i \text{ for } m \leq i \leq k\}.$$

The product measure $\mathbb{P}$ is a probability measure on $\Sigma$, called the Bernoulli measure. Furthermore, the measure $\mathbb{P}$ is characterized by
$$
\mathbb{P}([m; \alpha_m, \ldots, \alpha_k]) = \prod_{i=m}^{k} p_{\alpha_i}.
$$

\paragraph{Method 2 (Stochastic Version).}  Let $(r_k)_{k \in \mathbb{N}_0}$ be a sequence of random variables defined on the probability space $(\Sigma, \mathcal{A}, \mathbb{P})$, that is, a sequence $(r_k)$ taking values in $B$, sampled according to the uniform distribution and independently at each step.

Given an initial point $x_0 \in \mathbb{R}^n$ and a stepsize sequence $(\beta_k)_{k \in \mathbb{N}_0}$ of positive reals, the sequence $(x_k)$ is recursively defined by:
\begin{equation}\label{stochastic method}
x_{k+1} \in  {\rm Prox}_{\beta_k f_{r_k}}(x_k), \qquad k \in \mathbb{N}_0.
\end{equation}

\medskip

Our aim is to establish that, under suitable assumptions on the stepsize and on the regularity of the component functions, the sequence $(x_k)$ converges, almost surely, to a minimizer of the full objective function $f$.

\begin{proposition}\label{stochastic-descent}
 Let $ (x_k) $ be the sequence defined in  \eqref{stochastic method}, where the functions $f_1, \ldots, f_N$ are Lipschitz continuous with a common constant $L$ and prox-convex with modulus $\alpha > 0$. Then, for every $x \in  {\rm dom}\,f$, the following inequality holds:
 \begin{equation}\label{bound:sto} 
  \mathbb{E} \left[ \|x_{k+1} - x\|^2 \, \middle| \, \mathcal{F}_k \right] \leq \|x_k - x\|^2 - \frac{2 \beta_k}{\alpha N} \left[ f(x_k) - f(x) \right] + \dfrac{4\beta_k^2L^2}{\alpha},
 \end{equation} 
 where $\mathcal{F}_k := \sigma(r_0, r_1, \ldots, r_k)$.
\end{proposition}

\begin{proof}
Fix $x \in {\rm dom}\,f$ arbitrarily. For each $n \in \{1, \ldots, N\}$, define
$$ x_k^{(n)} := {\rm Prox}_{\beta_k f_n}(x_k).$$
Applying Lemma \ref{lemma1} to each $f_n$, we obtain
$$\|x_k^{(n)} - y\|^2 \leq \|x_k - x\|^2 - \dfrac{2\beta_k}{\alpha} \left[ f_n(x_k^{(n)}) - f_n(x) \right].$$
Since $x_{k+1} = x_k^{(r_k)}$, we have
$$\|x_{k+1} - x\|^2 \leq \|x_k - x\|^2 - \dfrac{2\beta_k}{\alpha} \left[ f_{r_k}(x_{k+1}) - f_{r_k}(x) \right].$$
Taking the conditional expectation with respect to $\mathcal{F}_k$,
$$ \mathbb{E} \left[ \|x_{k+1} - x\|^2 \,\middle|\, \mathcal{F}_k \right]  \leq \|x_k - x\|^2 - \dfrac{2\beta_k}{\alpha} \mathbb{E} \left[ f_{r_k}(x_{k+1}) - f_{r_k}(x) \middle|\, \mathcal{F}_k \right].$$
Using the fact that $r_k$ is independent of $\mathcal{F}_k$ and uniformly distributed over $\{1,\dots,N\}$, we obtain
$$
\mathbb{E} \left[ \|x_{k+1} - x\|^2 \,\middle|\, \mathcal{F}_k \right] 
\leq \|x_k - x\|^2 - \dfrac{2\beta_k}{\alpha} \cdot \frac{1}{N} \sum_{n=1}^N \left[ f_n(x_k^{(n)}) - f_n(x) \right].
$$
Next, we decompose each term as follows
$$
f_n(x_k^{(n)}) - f_n(x) = f_n(x_k) - f_n(x) - \left[f_n(x_k) - f_n(x_k^{(n)})\right].
$$
Hence,
$$
\sum_{n=1}^N \left[ f_n(x_k^{(n)}) - f_n(x) \right] = f(x_k) - f(x) - \sum_{n=1}^N \left[ f_n(x_k) - f_n(x_k^{(n)})\right].
$$
Using the Lipschitz continuity of each $f_n$,  together with inequality \eqref{equa9}, we have
$$
f_n(x_k) - f_n(x_k^{(n)}) \leq L \|x_k^{(n)}-x_k\| \leq 2\beta_k L^2.
$$
Summing over $n$, we obtain
$$
\sum_{n=1}^N \left[ f_n(x_k^{(n)}) - f_n(x_k) \right] \leq 2\beta_k L^2 N.
$$
Putting all together
$$\mathbb{E} \left[ \|x_{k+1} - x\|^2 \,\middle|\, \mathcal{F}_k \right] 
\leq \|x_k - x\|^2 - \frac{2 \beta_k}{\alpha N} \left[ f(x_k) - f(x) \right] + \dfrac{4 \beta_k^2 L^2}{\alpha},$$
which completes the proof.
\end{proof}

To prove the convergence of the sequence generated by Method~2,  we will use the following result, known as the Supermartingale Convergence Theorem. A proof of this result can be found, for instance, in \cite{Bertsekas-96} or \cite{Neveu-75}.

\begin{theorem}[Supermartingale Convergence Theorem]\label{thm:supermartingale}
Let $(\Omega, \mathcal{F}, (\mathcal{F}_k)_{k \in \mathbb{N}_0}, \mu)$ be a filtered probability space. Assume that $(Y_k)$, $(Z_k)$, and $(W_k)$ are sequences of nonnegative real-valued random variables defined on $\Omega$, and assume that:
\begin{itemize}
    \item[i.] $Y_k$, $Z_k$, and $W_k$ are $\mathcal{F}_k$-measurable for each $k \in \mathbb{N}_0$,
    \item[ii.] $\mathbb{E}(Y_{k+1} \mid \mathcal{F}_k) \leq Y_k - Z_k + W_k$ for each $k \in \mathbb{N}_0$,
    \item[iii.] $\sum_{k=0}^{\infty} W_k < +\infty$ almost surely.
\end{itemize}
Then the sequence $(Y_k)$ converges to a finite nonnegative random variable $Y$ almost surely, and $\sum_{k=0}^{\infty} Z_k < +\infty$ almost surely.
\end{theorem}

The desired result is given below.

\begin{theorem}
Let $ (x_l) $ be the sequence defined in Method 2, where $(\beta_k)_{k \in \mathbb{N}_0}$ is a sequence of positive real numbers satisfying

\begin{equation} \sum_{k=0}^{\infty} \beta_k = +\infty, \quad \sum_{k=0}^{\infty} \beta_k^2 < +\infty. \end{equation}
If ${\rm \arg \min}_{K} f \neq \emptyset$ and the functions $ f_1, \ldots, f_N $ are Lipschitz continuous and prox-convex with modulus $\alpha$. Then, with probability $1$, the sequence $(x_l)$ converges to a minimizer of $f$.
\end{theorem}

\begin{proof}
Given $x^{\ast} \in {\rm argmin}_{K}  f$, and using Proposition~\ref{stochastic-descent}, we obtain
\begin{equation*}
\mathbb{E} \left[ \|x_{k+1} - x^{\ast}\|^2 \,\middle|\, \mathcal{F}_k \right] \leq \|x_k - x^{\ast}\|^2 - \frac{2  \beta_k}{\alpha N} \left[ f(x_k) - f(x^{\ast}) \right] + \dfrac{4 \beta_k^2 L^2}{\alpha},
\end{equation*}
where $\mathcal{F}_k := \sigma(r_0, r_1, \ldots, r_k)$.

By Theorem~\ref{thm:supermartingale}, there exists a subset $\Sigma_{x^{\ast}} \subset \Sigma$ with $\mathbb{P}(\Sigma_{x^{\ast}}) = 1$ such that, for every $(r_k)_{k \in \mathbb{N}_0} \in \Sigma_{x^{\ast}}$, the sequence $\|x_k - x^{\ast}\|$ converges and
\begin{equation}\label{sum_bound4}
 \sum_{k=0}^{\infty} \frac{2  \beta_k}{\alpha N} \left[ f(x_k) - f(x^{\ast}) \right] < +\infty.
\end{equation}

Since $\mathbb{R}^n$ is separable, there exists a countable dense subset $\{w_i\}_{i \in \mathbb{N}} \subset {\rm argmin}_{K}  h$  that is dense in ${\rm argmin}{K} h$. Define
$$
\overline{\Sigma} := \bigcap_{i=1}^{\infty} \Sigma_{w_i}.
$$
Note that
$$
\mathbb{P}(\overline{\Sigma}) = \mathbb{P}\left( \bigcap_{i=1}^{\infty} \Sigma_{w_i} \right) = 1 - \mathbb{P}\left( \bigcup_{i=1}^{\infty} \Sigma_{w_i}^c \right) \geq 1 - \sum_{i=1}^{\infty} \mathbb{P}(\Sigma_{w_i}^c) = 1.
$$
Thus, $\mathbb{P}(\overline{\Sigma}) = 1$. Moreover, for every sequence $(r_k)_{k\in \mathbb{N}_0} \in \overline{\Sigma}$ and $i \in \mathbb{N}$, we have that $\lVert x_k - w_i\rVert$ is convergent and satisfies \eqref{sum_bound4}. Given that  $\sum_{k=0}^{\infty} \beta_k = +\infty$, it follows from \eqref{sum_bound4} that
\begin{equation}\label{liminf}
\liminf_{k \to \infty} f(x_k)= \min (f) \quad .
\end{equation}
Given $i \in \mathbb{N}$, the sequence $\lVert x_k - w_i\rVert$ is convergent and, in particular, bounded. Since $f$ is continuous, it follows from \eqref{liminf} that $(x_k)_{k\in \mathbb{N}_0}$ admits some cluster point $x^{\ast} \in \arg \min (f)$. Given $\epsilon > 0$, since $\{w_i\}_{i\in \mathbb{N}}$ is dense in ${\rm argmin}_{K}  f$, there exists $i(\epsilon)\in \mathbb{N}$ such that $\lVert w_{i(\epsilon)} - x^{\ast}\rVert < \epsilon$. Note that
\begin{equation*}
\lim_{k\to \infty} \lVert x_k - x_{i(\epsilon)}\rVert \leq \liminf_{k\to \infty}\left( \lVert x_k - x^{\ast}\rVert + \lVert x^{\ast} - x_{i(\epsilon)}\rVert \right) \leq \epsilon.
\end{equation*}
Hence,
\begin{equation}
\limsup_{k\to \infty} \lVert x_k - x^{\ast}\rVert \leq \lim_{k\to \infty}\left( \lVert x_k - x_{i(\epsilon)}\rVert + \lVert x^{\ast} - x_{i(\epsilon)}\rVert \right) \leq 2\epsilon.
\end{equation}
Since $\epsilon > 0$ is arbitrary, it follows that $(x_k)_{k\in \mathbb{N}_0}$ converges to $x^{\ast}$.
\end{proof}

Note that our stochastic framework extend the classical randomized PPA from convex functions (see \cite{Bacak, Bertsekas}) to prox-convex ones.

\begin{remark}
 Deterministic component selection strategies, such as cyclic, provide predictable behavior and convergence guarantees that hold for any admissible selection sequence, making them particularly robust and amenable to worst-case analysis. These features may be preferable in settings where reliability and reproducibility of the iterates are of primary importance. In contrast, stochastic selection rules typically rely on convergence guarantees established in expectation or almost surely and are often attractive in large-scale problems due to their reduced per-iteration computational cost, since only one component is processed at each iteration. Moreover, stochastic schemes may exhibit favorable convergence behavior on average, especially when evaluating all components at every iteration is impractical. Overall, the choice between stochastic and deterministic component selection reflects a trade-off between robustness and predictability versus lower per-iteration cost and potential gains in average performance, in line with classical results for convex optimization (see, e.g., \cite{Bertsekas}).    
\end{remark}

\section{Numerical Experiments}\label{sec:04}

In this section, we will provide some numerical illustrations as an application
of Method 1 and Method 2 for solving Problem (\ref{equa1}) defined in the previous subsection. The algorithm was coded in 
Matlab R2016a and run on a PC Intel(R) Core(TM) i5-2430M CPU @ 2.40 GHz  4GB Ram. We used the  Optimization Toolbox (fmincon, quadprog) to solve strongly convex and quadratic programming problems.  We stopped the program by using the stopping criterion 
$$\|x^{k+1}-x^k\|\leq \epsilon.$$
We consider Problem (\ref{equa1}), where 
$$
K:=\left\{x\in \mathbb R^n: 0\le x_1\le 2, \ 0 \leq x_i \leq 5 + \frac i {3i-2}, ~ \forall ~ i = 2, \hdots,n\right\},
$$
$N=5,$ $f(x)=\sum_{i=1}^Nf_i(x),$ $f_1(x):=-x_1^2-x_1+u^{\top}Au$,  $f_2(x):=x^{\top}Bx$, 
$f_3(x):=x^{\top}Cx$,
$f_4(x):=x^{\top}Px$, and $f_5(x):=x^{\top}Qx$,
 with $u=(x_2,\hdots,x_n)$ and $A=(a_{ij})_{(n-1)\times(n-1)}$, $B=(b_{ij})_{n\times n}$, $C=(c_{ij})_{n\times n}$ being  symmetric positive semidefinite matrices. By Example \ref{exam2.2} and Lemma \ref{Lem2.1}, we have that $f_i$, ($i=1,\ldots,5$), is prox-convex on $K$  with modulus $\alpha=1$ and satisfies (\ref{prox:all}) for all $\beta>0$. Clearly,  $f_1,\ldots,f_5$ are Lipschitz continuous with a common constant $L$.

\textbf{Test 1.} 
In this experiment, we apply Method 1 and Method 2 to solve Problem (\ref{equa1}) using the same data. We use the periodic version for Method 1.  \textbf{Figures} \ref{fig1}-\ref{fig3} present the convergence results of  Method 1 and Method 2 with respect to both the number of iterations and CPU time, in which the number of dimensions $n=100$ is fixed, and the error $\epsilon$ is different.

\begin{figure}[H]
	\begin{tabular}{cc}
		\includegraphics[width=6.1cm]{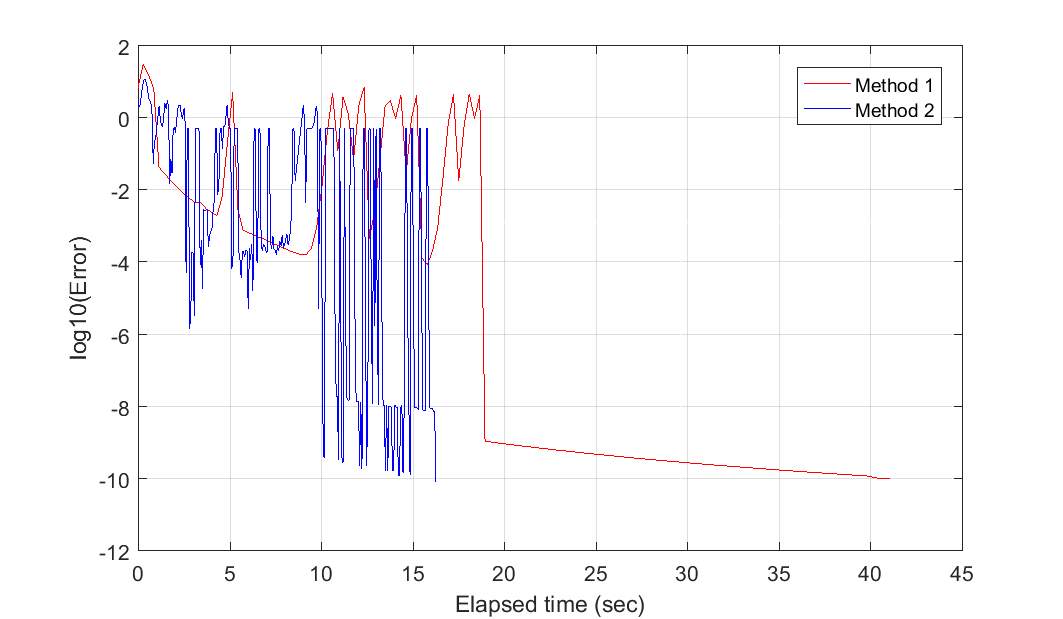} &
		\includegraphics[width=6.1cm]{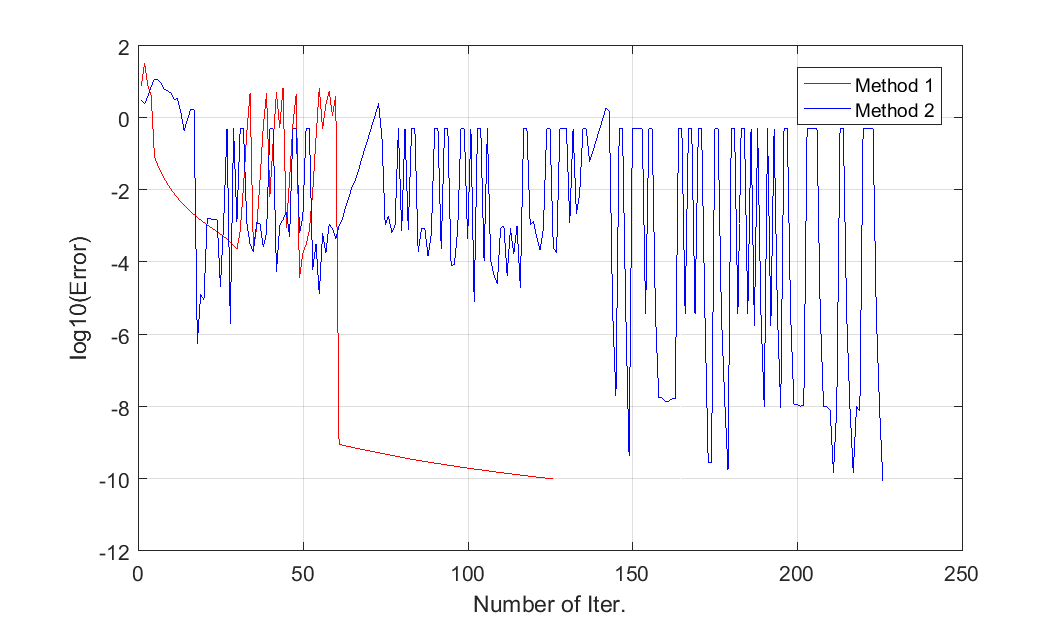}
	\end{tabular}
	\caption{Convergence of $\log_{10}(Error)$ of  Method 1 and Method 2 with respect to the number of iterations and CPU time(s), where $n=100$ and $\epsilon=10^{-10}$} \label{fig1}
\end{figure}

\begin{figure}[H]
	\begin{tabular}{cc}
		\includegraphics[width=6.1cm]{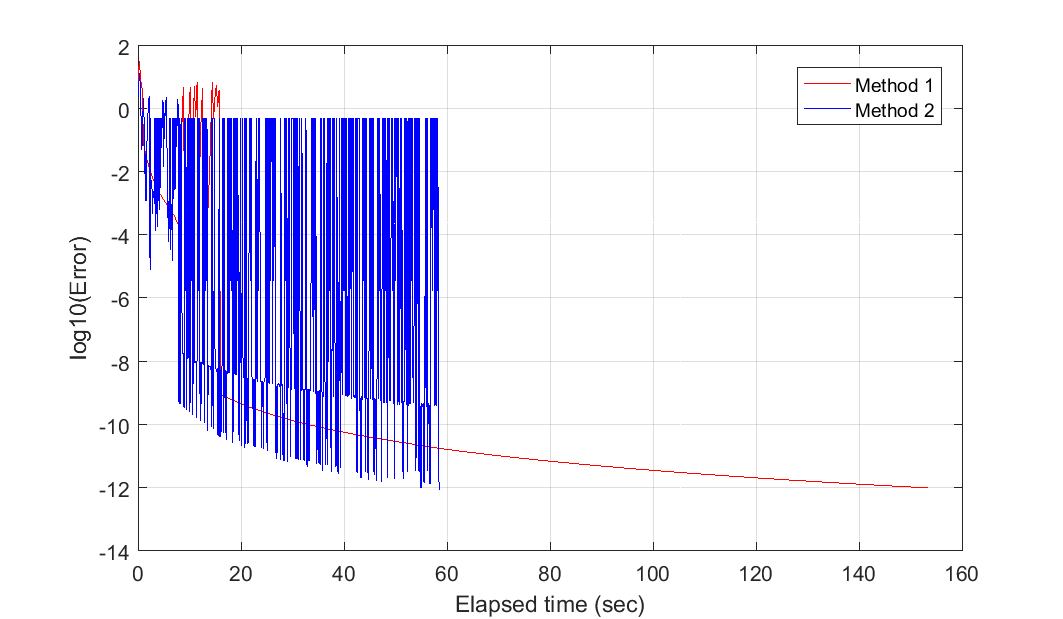} &
		\includegraphics[width=6.1cm]{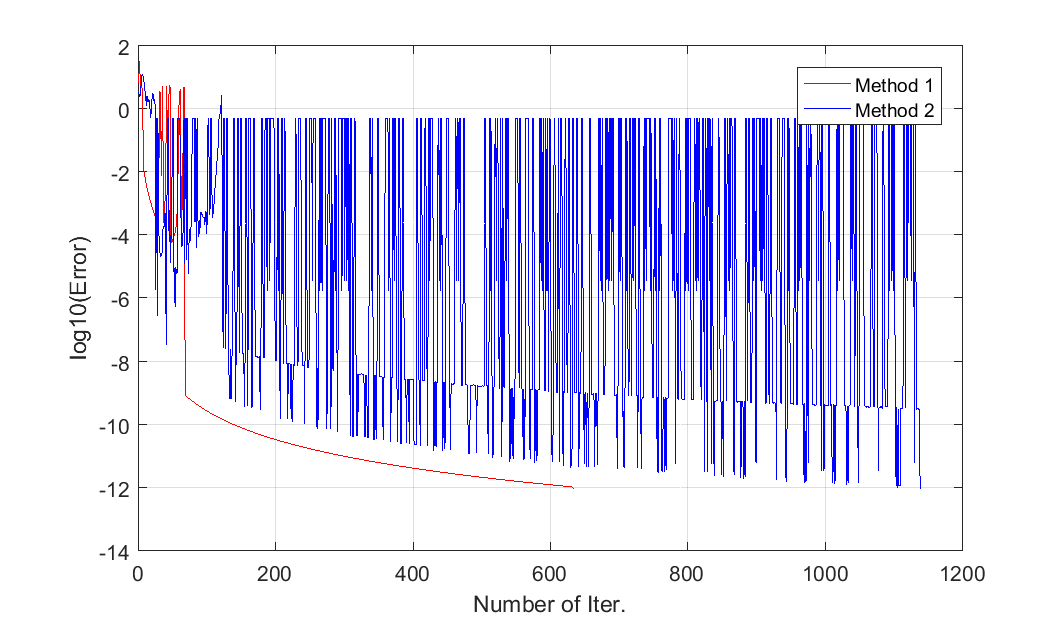}
	\end{tabular}
	\caption{Convergence of $\log_{10}(Error)$ of  Method 1 and Method 2 with respect to the number of iterations and CPU time(s), where $n=100$ and $\epsilon=10^{-12}$} \label{fig2}
\end{figure}

\begin{figure}[H]
	\begin{tabular}{cc}
		\includegraphics[width=6.1cm]{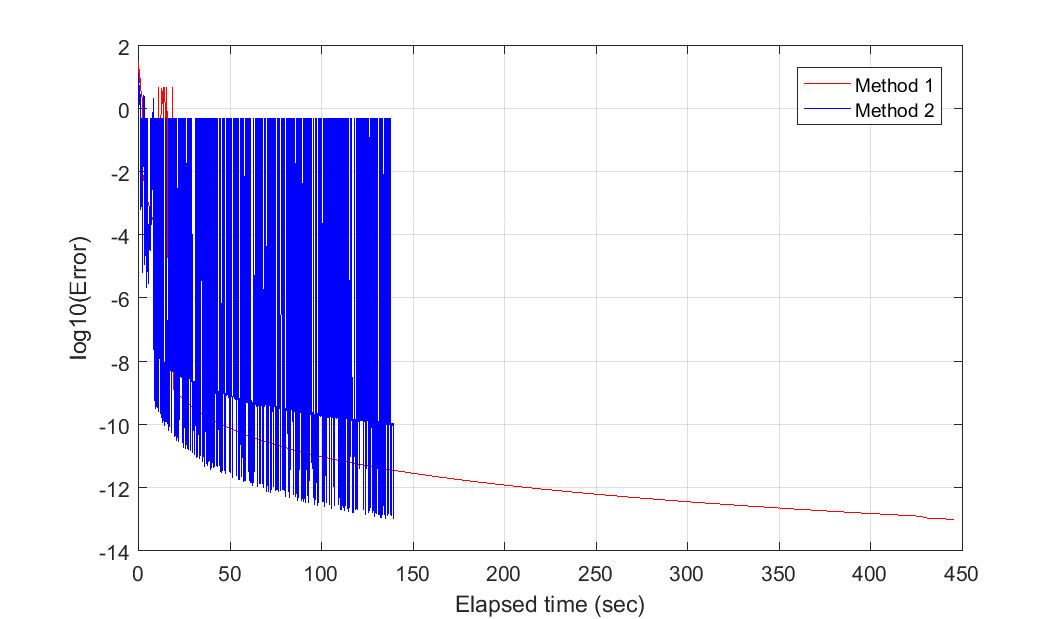} &
		\includegraphics[width=6.1cm]{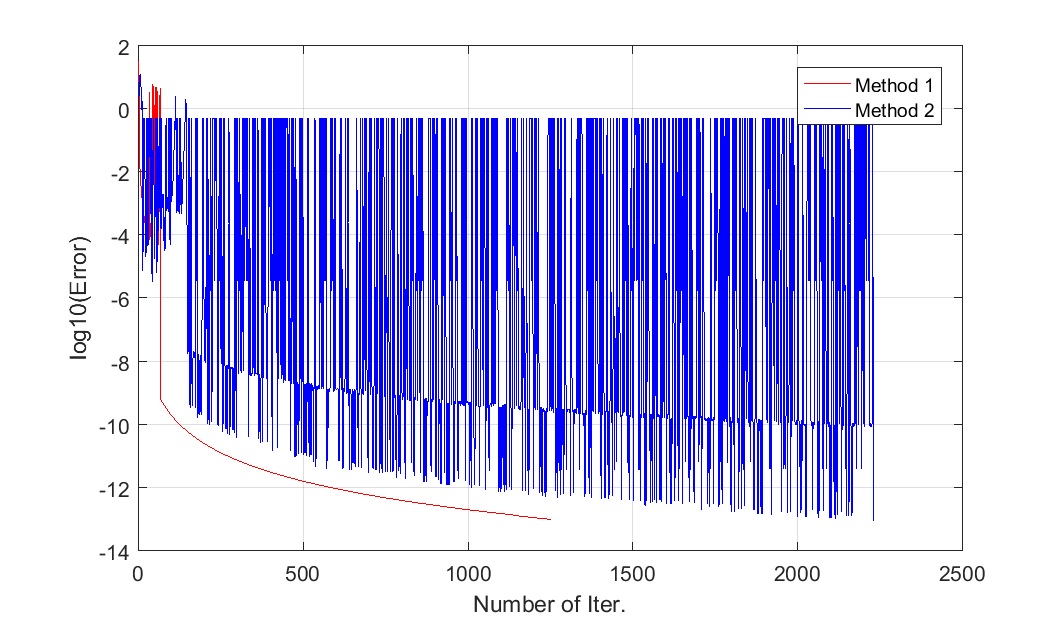}
	\end{tabular}
	\caption{Convergence of $\log_{10}(Error)$ of  Method 1 and Method 2 with respect to the number of iterations and CPU time(s), where $n=100$ and  $\epsilon=10^{-13}$} \label{fig3}
\end{figure}

\begin{figure}[H]
	\begin{tabular}{cc}
		\includegraphics[width=6.1cm]{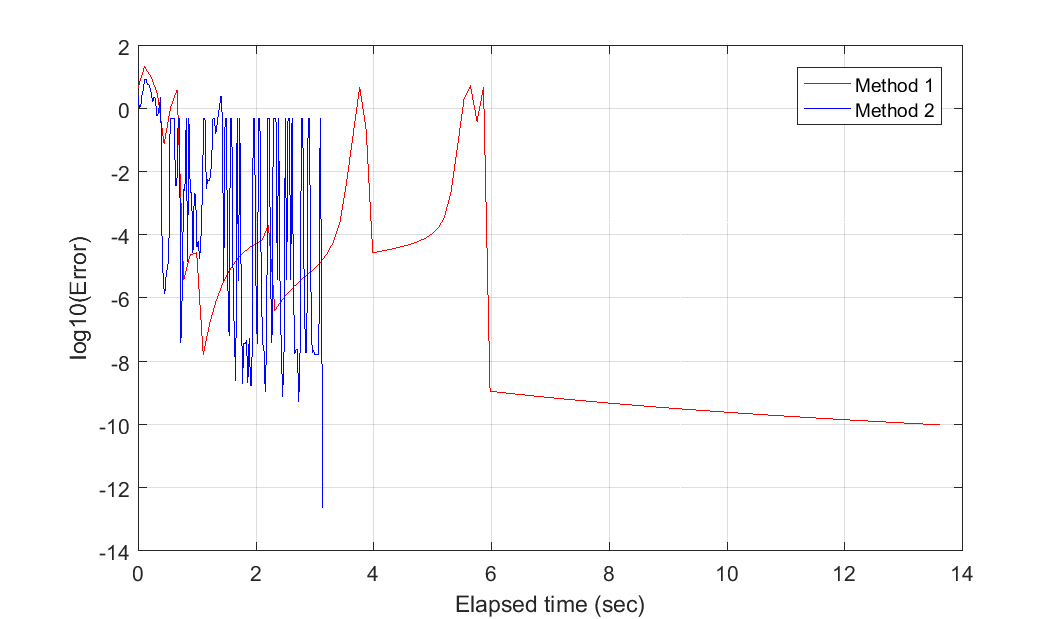} &
		\includegraphics[width=6.1cm]{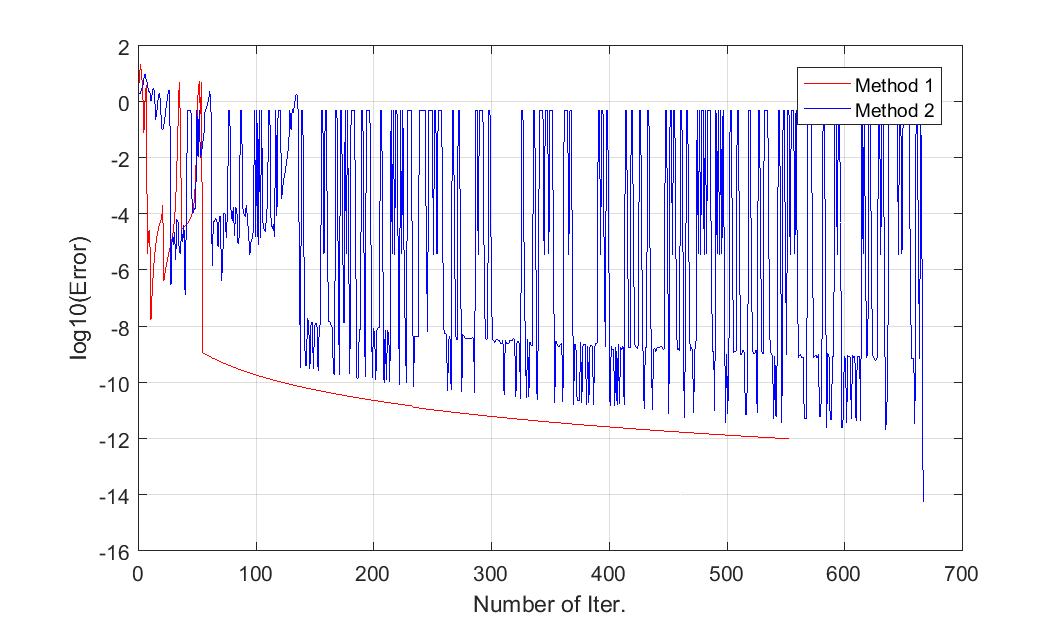}
	\end{tabular}
	\caption{Convergence of $\log_{10}(Error)$ of  Method 1 and Method 2 with respect to the number of iterations and CPU time(s), where $n=50$ and  $\epsilon=10^{-10}$} \label{fig4}
\end{figure}

\textbf{Figures} \ref{fig4}-\ref{fig7} present the convergence results of  Method 1 and Method 2 to both the number of iterations and CPU time, in which the error $\epsilon=10^{-10}$ is fixed, and the number of dimensions $n$ is different.
\begin{figure}[H]
	\begin{tabular}{cc}
		\includegraphics[width=6.1cm]{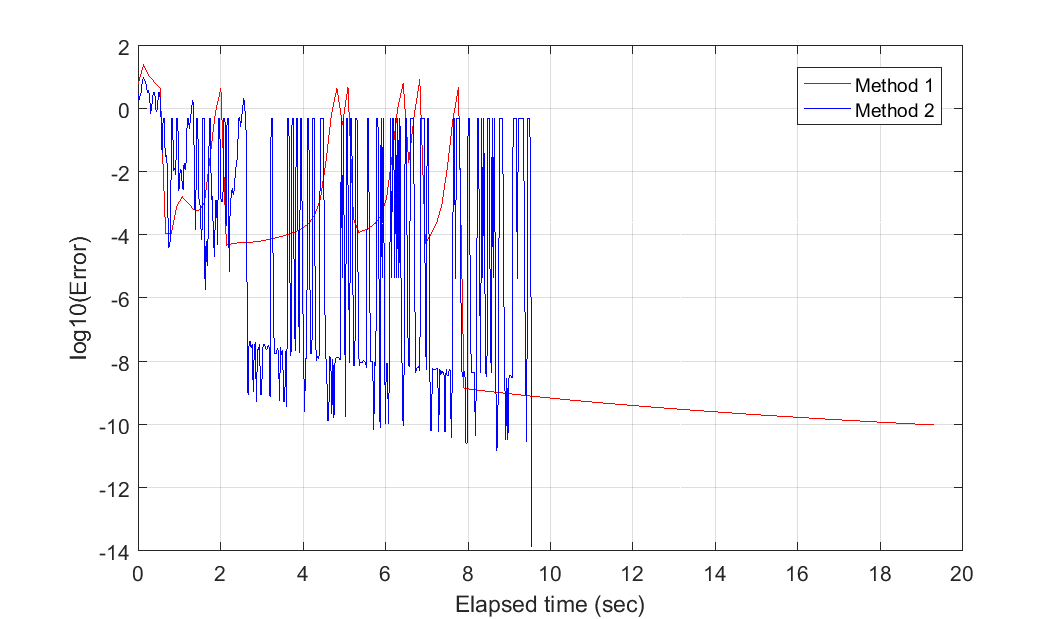} &
		\includegraphics[width=6.1cm]{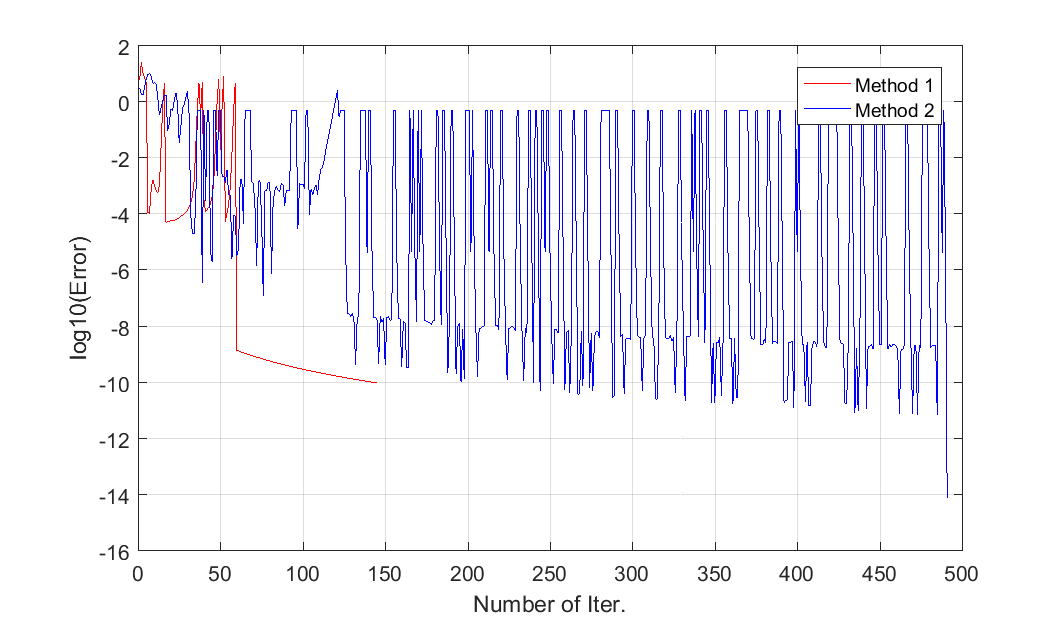}
	\end{tabular}
	\caption{Convergence of $\log_{10}(Error)$ of  Method 1 and Method 2 with respect to the number of iterations and CPU time(s), where $n=70$ and  $\epsilon=10^{-10}$} \label{fig5}
\end{figure}

\begin{figure}[H]
	\begin{tabular}{cc}
		\includegraphics[width=6.1cm]{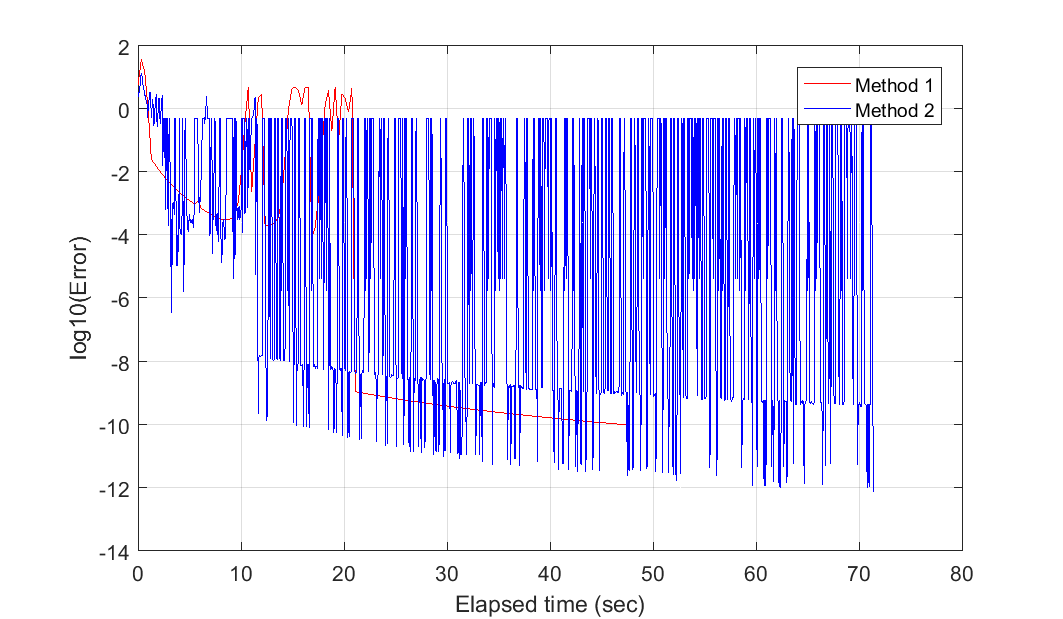} &
		\includegraphics[width=6.1cm]{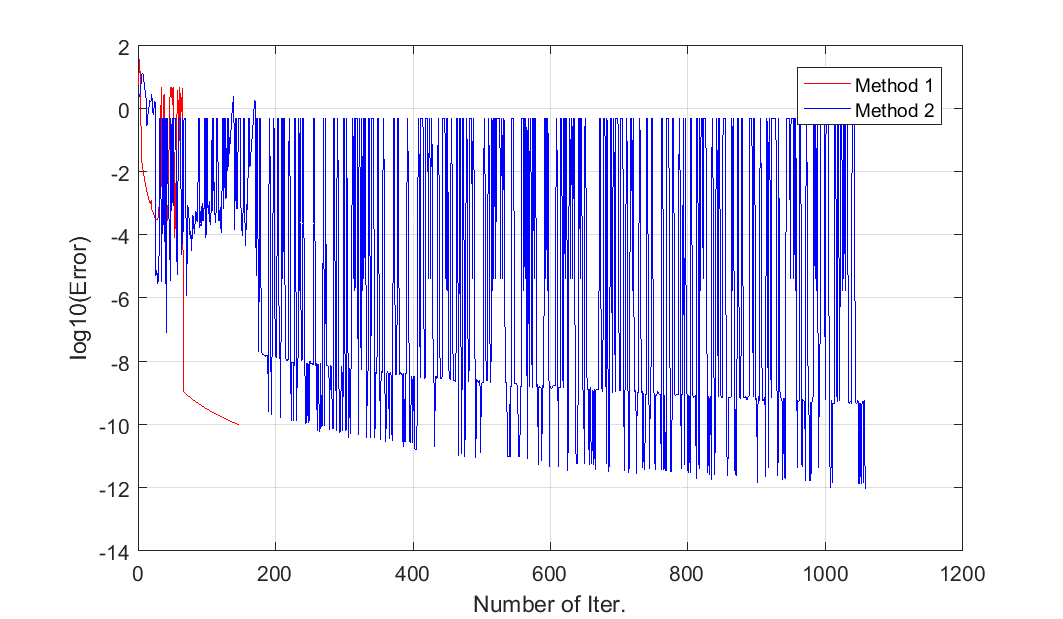}
	\end{tabular}
	\caption{Convergence of $\log_{10}(Error)$ of  Method 1 and Method 2 with respect to the number of iterations and CPU time(s), where $n=120$ and  $\epsilon=10^{-10}$} \label{fig6}
\end{figure}

\begin{figure}[H]
	\begin{tabular}{cc}
		\includegraphics[width=6.1cm]{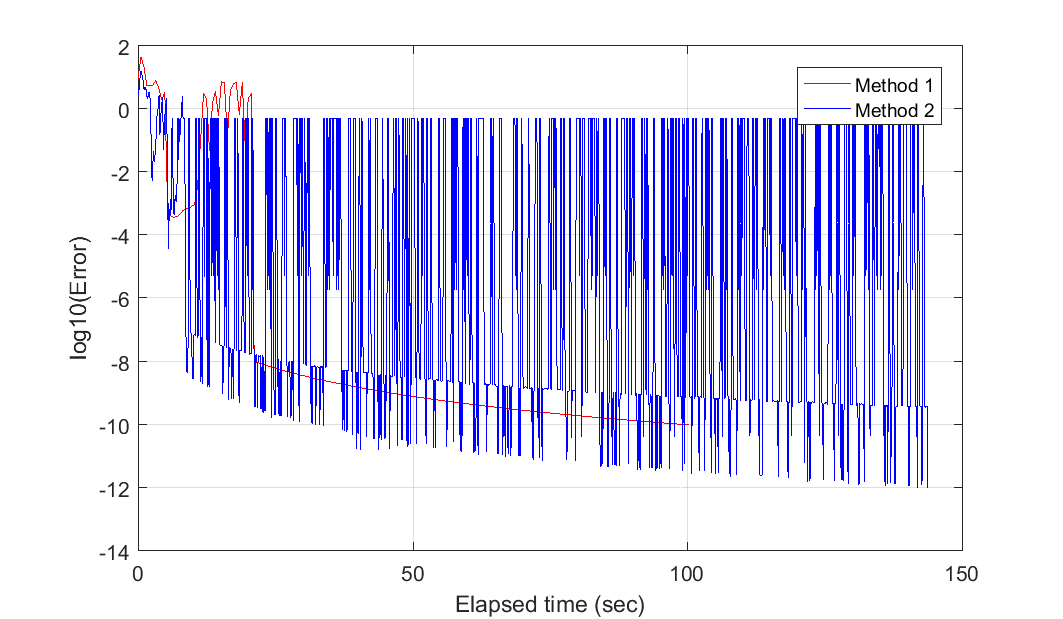} &
		\includegraphics[width=6.1cm]{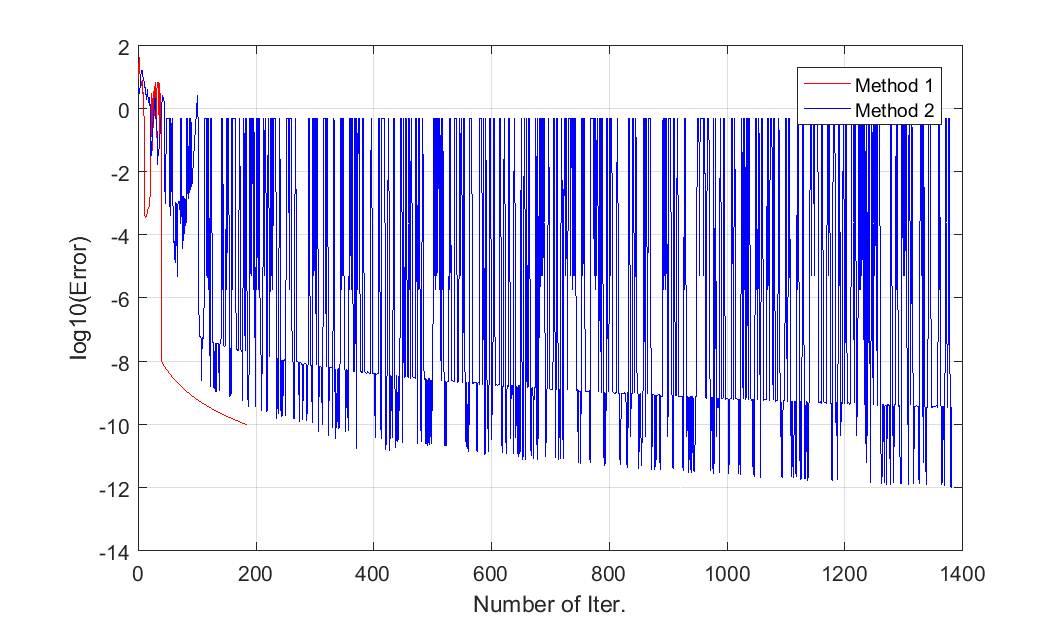}
	\end{tabular}
	\caption{Convergence of $\log_{10}(Error)$ of  Method 1 and Method 2 with respect to the number of iterations and CPU time(s), where $n=200$ and  $\epsilon=10^{-10}$} \label{fig7}
\end{figure}

\textbf{Test 2.} In this test, we employed the same data and initialization points as in \textbf{Test 1} and the matrices $A,\ B,\ C,\ P,\ Q$ randomly generated by using the command:
\begin{align*}
	& H=orth(randn(n-1)), D=diag(abs(randn(n-1,1))), A=H*D*H';\\ 
	&  H_1=orth(randn(n));\ D_1=diag(abs(randn(n,1))+0.3);\ B=H_1*D_1*H_1';\\
	& H_2=orth(randn(n));\ D_2=diag(abs(randn(n,1))+0.3);\ C=H_2*D_2*H_2';\\
	& H_3=orth(randn(n));\ D_3=diag(abs(randn(n,1))+0.3);\ P=H_3*D_3*H_3'; \\
	& H_4=orth(randn(n));\ D_4=diag(abs(randn(n,1))+0.3);\ Q=H_4*D_4*H_4'.
\end{align*} \textbf{Table 1} and \textbf{Table 2} report the experimental results of Method 1 and Method 2 for various choices of $n$ and $\epsilon$. We report the average number of iterations and the CPU time over 30 independent runs for each method.

\begin{table}[h]
	\begin{center} 
		\begin{tabular}{ |c|l l|l l|}
			\hline
			&\multicolumn{2}{|c|}{Method 1} & \multicolumn{2}{|c|}{Method 2}\\
			\cline{2-5}
			\ \ Dimension ($n$)\ \ \ & Iterations&\   CPU-time & Iterations & \  CPU-time \\
			\hline
			50&143.6
			&29.6732
			&447.2&8.7643
			\\ \hline
			70&158&46.3242
			&668.7&30.3461
			\\ \hline
			120&197.4&52.41516
			&863.6&41.15626\\
			\hline
			200&231.4&109.2435
			&1081.3&106.0937\\
			\hline
		\end{tabular}
	\end{center}
	\label{table1}
	\caption{Results of average iterations and average CPU time(s) of the methods with different dimensions and $\epsilon=10^{-10}$}
\end{table}

\begin{table}[h]
	\begin{center} 
	\begin{tabular}{ |c|l l|l l|}
		\hline
		&\multicolumn{2}{|c|}{Method 1} & \multicolumn{2}{|c|}{Method 2}\\
		\cline{2-5}
		\ \ Error ($\epsilon$)\ \ \ & Iterations &\   CPU-time & Iterations & \  CPU-time \\
		\hline
		$10^{-10}$&231.4&109.2435
		&1081.3&106.0937
		\\
		\hline
			$10^{-12}$&286.6&143.6632
			&1481.3&145.8743
		\\
		\hline
			$10^{-13}$ &435.8&205.2594
			&2121.3&207.1128\\
		\hline
	\end{tabular}
\end{center}
	\label{table2}
	\caption{Results of average iterations and average CPU time(s) of the methods with different errors and $n=200$}	
\end{table}

Preliminary findings from \textbf{Tests} 1 and 2 lead to the following observations:
\begin{itemize}
	\item[$(i)$]
 When the problem dimension and the required accuracy are low, Method 2 demonstrates better computational efficiency in terms of runtime compared to Method 1, but requires a greater number of iterations.
	\item[$(ii)$] As the problem size and required accuracy increase, the computational time difference between the two methods decreases.
\end{itemize}

\section{Conclusions}\label{sec:05}

We contributed to the extension of splitting proximal point type algorithms beyond convexity by proving the convergence of two different algorithms: a deterministic permutation-based and a random-order variant, ensuring global convergence for both methods under standard stepsize a\-ssump\-tions, with almost sure convergence for the stochastic variant via supermartingale theory. A key advantage of our framework is that it relies solely on the prox-convexity of the component functions, which does not imply that the full function $f$ is prox-convex itself.

Future work includes the development of accelerated and inertial variants, complexity analysis, and a randomized scheme in the spirit of the extensively studied Alternating Projections literature \cite{Amemiya65, BritoCruzMelo24, Hundal04, KopeckaPaszkiewicz17, Melo22, Prager60, Sakai95}.

\section{Declarations}

\subsection*{Conflict of interest statement}\label{sec6}
No potential conflict of interest was reported by the author.

\subsection*{Data Availability}\label{sec7}

The datasets generated and analyzed during the current study are available from all authors on reasonable request.

\subsection*{Acknowledgements} 
 This research was partially supported by ANID--Chile under project Fondecyt Regular 1241040 (Lara).

\end{document}